\documentclass[oneside,english]{amsart}
\usepackage[T1]{fontenc}
\usepackage[latin9]{inputenc}
\usepackage{color}
\usepackage{amsthm}
\usepackage{amssymb}

\makeatletter
\numberwithin{equation}{section}
\numberwithin{figure}{section}
\theoremstyle{plain}
\newtheorem{thm}{\protect\theoremname}
\theoremstyle{plain}
\newtheorem{lem}[thm]{\protect\lemmaname}
\newtheorem{rmk}[thm]{Remark}

\usepackage{babel}

\makeatother

\usepackage{babel}
\providecommand{\lemmaname}{Lemma}
\providecommand{\theoremname}{Theorem}

\begin{document}
\title[Neumann problem for a type of $k$-Hessian equations]{The Neumann problem for a type of fully nonlinear complex equations}
\author{Weisong Dong}
\address{School of Mathematics, Tianjin University, Tianjin, P.R.China, 300354}
\email{dr.dong@tju.edu.cn}
\author{Wei Wei}
\address{Shanghai Center for Mathematical Sciences, Fudan University, Shanghai,
P.R.China.}
\email{wei\_wei@fudan.edu.cn}

\begin{abstract}
In this paper we study the Neumann problem for a type of fully nonlinear
second order elliptic partial differential
equations on domains in $\mathbb{C}^{n}$ without any curvature
assumptions on the domain.
\end{abstract}


\maketitle
\section{Introduction}


Let $\Omega\subset\mathbb{C}^{n}$ be a smooth domain and
$\nu$ denote the unit outer
normal vector on $\partial\Omega$.
For a smooth function $u : \Omega \rightarrow \mathbb{R}$,
we denote the complex Hessian matrix of the function $u$ by $\partial\bar{\partial}u$.
Given
$\psi$
a smooth nonnegative function on $\bar{\Omega}\times\mathbb{R}$
and
$\varphi$ a smooth function on
$\partial\Omega\times\mathbb{R}$,
we study the following fully nonlinear second order partial differential equation
\begin{equation}
\label{eqn-1}
\sigma_{k}(\Delta u \texttt{I} -\partial\bar{\partial}u)=\psi(z,u),\; \mbox{in}\; \Omega, 
\end{equation}
with Neumann boundary data
\begin{equation}
u_{\nu} =\varphi(z,u), \;\mbox{on}\; \partial\Omega, \label{eqn-2}
\end{equation}
where $\Delta$ is the Laplace operator, $\texttt{I}$ is the $n\times n$ identity matrix
and $1\leq k\leq n-1$.
For any $n\times n$ Hermitian
matrix $A$, the function $\sigma_{k}$ for $1\leq k\leq n$ is defined as
$
\sigma_{k}(A)\equiv \sigma_{k}(\lambda(A))=
\sum 
\lambda_{i_{1}}\lambda_{i_{2}}\cdots\lambda_{i_{k}},
$
where the sum is over $\{1\leq i_{1}<\cdots<i_{k}\leq n\}$ and
$\lambda(A)=(\lambda_{1},\cdots,\lambda_{n})\in\mathbb{R}^{n}$
are the eigenvalues of $A$. For convenience, we denote $\sigma_{0}(\lambda)=1$.
Recall the G{\aa}rding cone
$
\Gamma_{k}=\{\lambda\in\mathbb{R}^{n}:\sigma_{i}(\lambda)>0,\;\forall1\leq i\leq k\}.
$
We say that a function $u\in C^{2}(\Omega)$ is $k$-\emph{admissible} if
$\lambda(\Delta u \texttt{I}-\partial\bar{\partial}u) (z) \in\Gamma_{k}$ for every $z\in \Omega$.
It is well known that
at a diagonal matrix $A = (a_{ij})$, $\frac{\partial \sigma_k}{\partial a_{ij}} = \frac{\partial \sigma_k}{\partial \lambda_i} \delta_{ij}$,
see Caffarelli-Nirenberg-Spruck \cite{CNS}.
Hence, equation \eqref{eqn-1} is elliptic for $k$-admissible
solutions, see Lemma \ref{key} in Section 2.

Fully nonlinear equations similar to \eqref{eqn-1} with the particular Hermitian
matrix $\Delta u \texttt{I}-\partial\bar{\partial}u$ inside the $\sigma_{k}$
have been studied extensively. Such type of equations first comes from
the study of Gauduchon conjecture. In \cite{STW} Sz\'ekelyhidi-Tosatti-Weinkove
solved a class of very general fully nonlinear equations on compact Hermitian manifolds, thereby the
Gauduchon conjecture was proved as a special case of their equation.
Actually, the Hermitian matrix inside the equation in  \cite{STW} also involved a linear term of the
gradient of $u$. Later on Tosatti-Weinkove \cite{TW19}
studied the Monge-Amp\`ere type equation with a general linear
gradient term inside the determinant on compact Hermitian manifolds
of complex dimension $n\geq2$. See also Yuan \cite{Yuan} for the
case of dimension 2. The Dirichlet problem for general fully nonlinear equations as in \cite{STW} has
been investigated by Feng-Ge-Zheng \cite{FGZ} on compact Hermitian manifolds with boundaries. Chu-Jiao studied $\sigma_k$ type of curvature equations  in \cite{CJ} and Chen-Tu-Xiang studied the Hessian quotient type of curvature equations \cite{CTX}.

After the real Monge-Amp\`ere equation was solved by Lions-Trudinger-Urbas in \cite{LTU}, the Neumann problem for complex Monge-Amp\`ere equation on complex domains
was studied by Li \cite{LiSY}. After that, there are
few progresses on the Neumann problem for the complex fully nonlinear equations. Until recently, following the breakthrough
work of Ma-Qiu \cite{MQ}, who solved the Neumann problem for $k$-Hessian
equations on uniformly convex domains in $\mathbb{R}^{n}$, some literatures
have appeared, such as Chen-Wei \cite{CW} for the complex Hessian
quotient equations and Chen-Chen-Mei-Xiang \cite{CCMX2} for a class of
mixed complex Hessian equations. See \cite{CZ} and \cite{CCX1} for
the corresponding problems in the real situations. Comparing to the complex fully nonlinear equations, the works about the real fully nonlinear equations are more abundant.  We refer to \cite{C07, JLL07, Jin, C09,  HS1, HS2, SY1, SY2, JT15, JT16, CMW, CMW2, DB1, DB2, HJ} and references therein. We remark that
in the above mentioned papers, the domain is assumed to be strictly
pseudoconvex, convex or mean convex. Due to the special structure of our equation,
we do not need the curvature assumption on the domain. As far as the authors know,
without any curvature conditions on the domain, it may be the first result
about the existence of a solution to the Hessian type fully nonlinear equations.
Namely, we have
\begin{thm}
\label{thm} Suppose $\Omega\subset\mathbb{C}^{n}$ is a
domain with $C^{4}$ boundary $\partial \Omega$. Let $\phi\in C^{3}(\partial\Omega)$ and $0\leq \psi\in C^{2}(\overline{\Omega})$
be given. Then there exists a unique $k$-admissible solution $u\in C^{3,\alpha}(\overline{\Omega})$
for $0 < \alpha <1$
to the Neumann problem
\begin{equation}
\begin{cases}
\sigma_{k}(\Delta u \texttt{I} -\partial\bar{\partial}u)=\psi(z), & \mbox{\text{in}}\;\Omega,\\
u_\nu = - \beta u + \phi(z), & \mbox{on}\;\partial\Omega,
\end{cases}\label{eqn-3}
\end{equation}
for any $1\leq k\leq n-1$ and $\beta $ a positive constant, where $\nu$ is the unit outer normal vector on $\partial\Omega$.
\end{thm}

For $k=1$ the equation is just the Poisson equation and we can drop the assumption $\psi \geq 0$.
For $k\geq 2$, the equation is fully nonlinear and nonnegativity of $\psi$ is necessary for the sake of $k$-admissible solutions.
The main purpose of this paper is to establish the a priori estimates for $k$-admissible solutions,
and then the existence of a solution follows by the continuity method
and uniqueness follows by the maximum principle.
The methods here also work for real equations. Deleting the bar of index in the proof of Theorem \ref{thm}, the proof is  same.  So the real counterpart of equation \eqref{eqn-3} is
also solvable.

\begin{thm}
\label{thm-real part} Suppose $\Omega\subset\mathbb{R}^{n}$ is a
domain with $C^{4}$ boundary $\partial \Omega$. Let $\phi\in C^{3}(\partial\Omega)$ and $0\leq \psi\in C^{2}(\overline{\Omega})$
be given. Then there exists a unique $k$-admissible solution $u\in C^{3,\alpha}(\overline{\Omega})$
for $0 < \alpha <1$
to the Neumann problem
\begin{equation}
\begin{cases}
\sigma_{k}(\Delta u \texttt{I} -D^2u)=\psi(z), & \mbox{\text{in}}\;\Omega,\\
u_\nu = - \beta u + \phi(z), & \mbox{on}\;\partial\Omega,
\end{cases}\label{eqn-3}
\end{equation}
for any $1\leq k\leq n-1$ and $\beta $ a positive constant, where $\nu$ is the unit outer normal vector on $\partial\Omega$.
\end{thm}

Using Theorem \ref{thm} and the approximation argument in Lions-Trudinger-Urbas \cite{LTU} (or
Qiu-Xia \cite{QX} ), with the additional assumption that $\Omega$ is strictly convex, one can easily show the following
theorem as Chen-Wei \cite{CW} whose proof is omitted here.
\begin{thm}
\label{thm2} Suppose $\Omega\subset\mathbb{C}^{n}$ is a $C^{4}$
strictly convex domain. Let $\phi\in C^{3}(\partial\Omega)$ and $0 \leq \psi \in C^{2}(\overline{\Omega})$
be given. Then there exist a unique constant $c$ and a unique $k$-admissible
solution $u\in C^{3,\alpha}(\overline{\Omega})$ (up to a constant) for $0<\alpha <1$ to the Neumann problem
\begin{equation}
\begin{cases}
\sigma_{k}(\Delta uI-\partial\bar{\partial}u)=\psi(z), & \mbox{in}\;\Omega,\\
u_\nu = c+\phi(z), & \mbox{on}\;\partial\Omega,
\end{cases}\label{eqn-4}
\end{equation}
for any $1\leq k\leq n-1$, where $\nu$ is the unit
outer normal vector on $\partial\Omega$.
\end{thm}

\begin{rmk}
The strict convexity here is used to get the uniform gradient estimate and overcome the possible blowing up
of the $C^0$ estimate as $\varepsilon \rightarrow 0$ for the solution $u_\varepsilon$ to equation \eqref{eqn-1} with
the boundary data $u_\nu = -\varepsilon u+\phi(z)$ on $\partial \Omega$.


\end{rmk}

We remark here that in the above two theorems the function $\psi$ can be zero at some (or all) points in $\Omega$.
This is due to our key Lemma \ref{key} in Section 2, which is equivalent to say that the equation is strictly elliptic
for $k$-admissible solutions.
Hence, all the a priori estimates derived here are independent of the lower bound of $\psi$,
which enable us to prove the existence result even for $\psi \geq 0$.
The key lemma still holds for $k+1$-admissible solutions to the quotient type equations,
i.e. $\frac{\sigma_k}{\sigma_l}$ type where $1\leq l < k \leq n-1$.
See Lemma \ref{key-2}.
All the a priori estimates for quotient type equations can be established similarly as for the equation \eqref{eqn-3}
using Lemma \ref{key-2} instead of Lemma \ref{key}.
Suppose $\Omega$, $\phi$ and $\psi$ are as before.
Therefore, we have

\begin{thm}\label{Hessian quotient estimates}
For any $1 \leq l < k\leq n-1$ and $\beta $ a positive constant,
suppose $u\in C^{4}(\overline{\Omega})$ is a $k+1$-admissible solution
to the Neumann problem
\begin{equation}
\begin{cases}
\frac{\sigma_{k}(\Delta u \texttt{I} -\partial\bar{\partial}u) }{\sigma_l (\Delta u \texttt{I} -\partial\bar{\partial}u)} =\psi(z), & \mbox{in}\;\Omega,\\
u_\nu = - \beta u + \phi(z), & \mbox{on}\;\partial\Omega,
\end{cases}\label{eqn-6}
\end{equation}
where $\nu$ is the unit outer normal vector on $\partial\Omega$. Then, we have the following a priori estimates
\[
 |u|_{C^2 (\Omega)} \leq C,
\]
where $C$ depends on $\psi, \phi, \Omega, n,\beta$ and $ k.$
\end{thm}

\begin{rmk}
The reason here we don't have existence of a solution immediately from the a priori estimate is that
we don't know whether the solutions are always in $\Gamma_{k+1}$ cone when one applies the continuity method.
\end{rmk}

We give a brief outline of the proof of Theorem \ref{thm}. We prove it
by the standard continuity method. So we shall derive a priori estimates for $k$-admissible
solutions up to second order derivatives, from which we can obtain
$C^{2,\alpha}$ estimates by Evans-Krylov Theorem and more higher
estimates by Schauder theory. The $C^0$ estimate for $k$-admissible
solutions can be derived easily by the maximum principle, see Lions-Trudinger-Urbas
\cite{LTU}. The gradient estimate is proved combining the interior
estimate and near boundary estimate, i.e. estimate on $\Omega_{\mu}=\{z\in\Omega|d(z,\partial\Omega)<\mu\}$
for sufficiently small $\mu>0$. The second order estimate is the
most difficult part. We first reduce the global second order estimate
to the estimate of the double normal derivative on the boundary. Then, we construct
barrier functions on the boundary strip $\Omega_{\mu}$ to derive
upper and lower bounds for the double normal derivative on $\partial\Omega$.

The rest of this paper is organized as follows. In Section 2, we recall
some properties of the elementary symmetric function $\sigma_{k}$
and prove two key lemmas, Lemma \ref{key} and Lemma \ref{key-2}. In Section 3, we prove the gradient estimates,
and the second order estimates are derived in Section 4.
In Section 5, we prove the existence of a solution.

\textbf{Acknowledgements}: The first named author is supported by
the National Natural Science Foundation of China, No.11801405 and
No.62073236. The second named author is supported by BoXin Programme,
BX20190082. Both authors would like to thank Prof. ChuanQiang Chen for careful reading and helpful suggestions.

\section{Preliminary}

We use $\sigma_{k}(\lambda|i)$ to denote the $k$-th elementary symmetric
function with $\lambda_{i}=0$ and $\sigma_{k}(\lambda|ij)$ the $k$-th
elementary function with $\lambda_{i}=\lambda_{j}=0$. For some useful
and well known properties of the elementary symmetric function , see
Li \cite{YYLi}, Lin-Trudinger \cite{LT}, Hou-Ma-Wu \cite{H-M-W}
and Huisken-Sinestrari \cite{HS}. 
Recall the following result, whose proof can be found in \cite{LT}.
\begin{lem}
\label{LT}
Let $\eta = (\eta_1, \cdots, \eta_n) \in\Gamma_{k}$ and $1\leq k\leq n$. Suppose that
\[
\eta_{1}\geq \eta_2 \geq \cdots\geq\eta_{n}.
\]
Then, we have
\[
\sigma_{k-1}(\eta|k)\geq c_{n,k}\sum_i \sigma_{k-1}(\eta|i),
\]
where $c_{n,k}$ is a positive constant only depending on $n$ and
$k$.
\end{lem}

The following generalized Newton-MacLaurin inequality is also used and its proof can be
found in Spruck \cite{Spruck}.

\begin{lem}
\label{NewMac}
For $\lambda \in \Gamma_k$, $n \geq k>l\geq 0$, $n \geq r> s \geq 0$, $k\geq r$ and $l\geq s$, we have
\[
\Big(\frac{\sigma_{k}(\lambda) / C_n^k }{\sigma_{l} (\lambda) / C_n^l }\Big)^{\frac{1}{k-l}}
\leq \Big(\frac{\sigma_{r} (\lambda) / C_n^r }{\sigma_{s} (\lambda) / C_n^s }\Big)^{\frac{1}{r-s}},
\]
and the equality holds if and only if $\lambda_1 = \cdots = \lambda_n$.
\end{lem}

Let $z=(z_{1},\cdots,z_{n})$ be a point in $ \mathbb{C}^{n}$. Sometimes we also
write $z$ in real coordinates as $z=(t_{1},\cdots,t_{2n})$ and $z_i=t_i+\sqrt{-1}t_{n+i}$. Given
$\xi\in\mathbb{R}^{2n}$, $D_{\xi}u$ denote the directional derivative
of $u$ along $\xi$. For the complex variables, we use the following
notations:
\[
\partial_{k} u =\frac{\partial u}{\partial z_{k}},\; \partial_{\bar{k}} u = \frac{\partial u}{\partial z_{\bar{k}}},\;
\partial_{i\bar{j}} u =\frac{\partial^{2}u}{\partial z_{i}\partial z_{\bar{j}}},\;
\partial_{i\bar{j}k} u =\frac{\partial^{3}u}{\partial z_{i}\partial z_{\bar{j}}\partial z_{k}},\;\cdots
\]
For simplicity, we write $u_{i}=\partial_{i}u$, $u_{i\bar{j}}=\partial_{i\bar{j}}u$,
$u_{i\bar{j}k}=\partial_{i\bar{j}k}u$, and so on.
It holds that 
\begin{equation}
\label{grad-equality}
|\nabla u|^{2}:=\sum_{j=1}^{n} \partial_{j} u \overline{\partial_{j} u}=\frac{1}{4}|D u|^{2}.
\end{equation}
In complex coordinates,
let $\eta_{i\bar{j}} \equiv \Delta u\delta_{ij}-u_{i\bar{j}}$.
With our notations, equation
\eqref{eqn-1} can be written as
\begin{equation}
\label{eqn'}
F(u_{i\bar{j}}) \equiv G(\eta_{i\bar{j}}) \equiv \sigma_{k}^{\frac{1}{k}}(\eta_{i\bar{j}})=\tilde{\psi}(z,u),
\end{equation}
where $\tilde{\psi}=\psi^{1/k}$.
Let $\lambda\equiv\lambda(u_{i\bar{j}})=(\lambda_{1},\cdots,\lambda_{n})$
be the eigenvalues of $\{u_{i\bar{j}}\}$ and $\eta\equiv\lambda(\eta_{i\bar{j}})=(\eta_{1},\cdots,\eta_{n})$
be the eigenvalues of $\{\eta_{i\bar{j}}\}$. Actually, $\eta_{i}=\sum_{j=1}^n\lambda_{j}-\lambda_{i}$.
Equation \eqref{eqn-1} can also be written as
\begin{equation}
f(\lambda) \equiv \sigma_{k}^{ \frac{1}{k} }(\eta)=\tilde{\psi}(z, u).\label{eqn''}
\end{equation}
We also introduce the following notations
\[
F^{i\bar{j}}=\frac{\partial F}{\partial u_{i\bar{j}}},\;\;
F^{i\bar{j},k\bar{l}}=\frac{\partial^{2}F}{\partial u_{i\bar{j}}\partial u_{k\bar{l}}},\;\;
G^{i\bar{j}}=\frac{\partial G}{\partial\eta_{i\bar{j}}},\;\;
G^{i\bar{j},k\bar{l}}=\frac{\partial^{2}G}{\partial\eta_{i\bar{j}}\partial\eta_{k\bar{l}}}.
\]
Recall that at a point where $\{u_{i\bar{j}}\}$ is diagonal, the
following identity holds
\[
F^{i\bar{j}}=f_{i}\delta_{ij},\;\mbox{ where }\;f_{i}=\frac{\partial f}{\partial\lambda_{i}}.
\]
This means $f_1, \cdots, f_n$ are the eigenvalues of $\{F^{i\bar j}\}$.

In the following, we always denote the eigenvalues of $\{u_{i\bar{j}}\}$
by $\lambda=(\lambda_{1},\cdots,\lambda_{n})$ with the ordering $\lambda_{1}\geq\cdots\geq\lambda_{n}$.
Then, $\eta=(\eta_{1},\cdots,\eta_{n}) \in \Gamma_k$ the eigenvalues of $\{\eta_{i\bar{j}}\}$
are ordered as $\eta_{1}\leq\cdots\leq\eta_{n}$, which implies
$\sigma_l (\eta | 1) \geq \cdots \geq \sigma_l (\eta | n)$ for $1\leq l \leq k-1$.
Therefore, at a point $z \in \Omega$
where $\{u_{i\bar{j}} (z) \}$ is diagonal, we have
\[
F^{i\bar{i}}=\sum_{k=1}^{n}G^{k\bar{k}}-G^{i\bar{i}} \; \mbox{and} \;
F^{1\bar{1}}\leq\cdots\leq F^{n\bar{n}}
\]
since
$G^{i\bar i} = \frac{1}{k} \sigma_k^{\frac{1}{k} - 1} \sigma_{k-1} (\eta | i)$
and $G^{1\bar{1}}\geq\cdots\geq G^{n\bar{n}}$.
Furthermore, by the fact that $\lambda_{i}=\frac{1}{n-1}\sum_{j=1}^{n}\eta_{j}-\eta_{i}$
and $\sigma_k^{\frac{1}{k}}$ is homogeneous of degree one,
we have
\begin{equation}
\sum_{i=1}^{n}F^{i\bar{i}}u_{i\bar{i}}=\sum_{i=1}^{n}G^{i\bar{i}}\eta_{i\bar{i}} = \tilde \psi \label{eq:fundamental equality 1}
\end{equation}
and
\begin{equation}
\sum_{i=1}^{n}F^{i\bar{i}}u_{i\bar{i}\xi\xi}=\sum_{i=1}^{n}G^{i\bar{i}}\eta_{i\bar{i}\xi\xi}.\label{eq:fundamental equality 2}
\end{equation}

By Lemma \ref{LT} and Lemma \ref{NewMac}, we can prove the following lemma which is
the key to our estimates.
\begin{lem}
\label{key}
Let $r$ be a $n\times n$ Hermitian matrix,
$\lambda = \lambda(r)$ be the eigenvalues of $r$ and $\eta \in \Gamma_{k}$ where
$\eta_i = \sum_j \lambda_j - \lambda_i$.
Suppose $0\leq f_{1}\leq\cdots\leq f_{n}$ are the eigenvalues
of $\{F^{i\bar j}(r)\}$ for (\ref{eqn''}).
If $1\leq k\leq n-1$, we have
\[
f_{1}\geq c_{n,k}\sum_{i=1}^{n}f_{i}
\]
for some uniform positive constant $c_{n,k}$ only depending on $n$ and $k$.
\end{lem}

\begin{proof}
We only need to prove the lemma when the matrix $r$ is diagonal. Without loss of generality,
we assume $\lambda_{1}(r)\geq\cdots\geq\lambda_{n}(r)$.
Hence, $\eta_{1}\leq\cdots\leq\eta_{n}$.
By $f(\lambda(r))=\sigma_{k}^{\frac{1}{k}}(\eta)$,
we see
\[
f_{i}=\frac{1}{k}\sigma_{k}^{\frac{1}{k}-1}\sum_{l\neq i}\sigma_{k-1}(\eta|l).
\]
Since $\eta_{1}\leq\cdots\leq\eta_{n}$ and $\eta\in\Gamma_{k}$,
we know
\[
\sigma_{k-1}(\eta|1)\geq\sigma_{k-1}(\eta|2)\geq\cdots\geq\sigma_{k-1}(\eta|n).
\]
By $1\leq k\leq n-1$ and Lemma \ref{LT},  we have $\sigma_{k-1}(\eta|2)\geq c_{n,k}\sum_{i=1}^{n}\sigma_{k-1}(\eta|i)$.
Hence
\[
f_{1}\geq\frac{1}{k}\sigma_{k}^{\frac{1}{k}-1}\sigma_{k-1}(\eta|2)\geq\frac{c_{n,k}}{k}\sigma_{k}^{\frac{1}{k}-1}\sum_{i=1}^{n}\sigma_{k-1}(\eta|i)=\frac{c_{n,k}}{n-1}\sum_{i=1}^{n}f_{i}
\]
where in the last inequality we used $\sum_{i=1}^{n}f_{i}=\frac{n-1}{k}\sigma_{k}^{\frac{1}{k}-1}\sum_{i=1}^{n}\sigma_{k-1}(\eta|i)$.
\end{proof}

For the quotient type equation \eqref{eqn-6}, we can prove a similar result for $k+1$-admissible solutions.
First, we can rewrite equation \eqref{eqn-6} as
\begin{equation}
F(u_{i\bar{j}}) \equiv \Big(\frac{\sigma_{k}}{\sigma_{l}}\Big)^{\frac{1}{k-l}}(\eta_{i\bar{j}})=\tilde{\psi},\;\label{eqn-quotient-1}
\end{equation}
where $1\leq l<k\leq n-1$ and $\tilde{\psi}=\psi^{\frac{1}{k-l}}$.
Using our notations, the above equation can be written equivalently
as
\begin{equation}
f(\lambda)\equiv \Big(\frac{\sigma_{k}}{\sigma_{l}}\Big)^{\frac{1}{k-l}}(\eta)=\tilde \psi,\label{eqn-quotient-2}
\end{equation}
where $\eta\in \Gamma_{k+1}$ and $\eta_{i}=\sum\lambda_{j}-\lambda_{i}$.
Similar to Lemma \ref{key}, we have
\begin{lem}
\label{key-2}
Let $r$ be a $n\times n$ Hermitian matrix,
$\lambda=\lambda(r)$ be the eigenvalues of $r$ and $\eta \in \Gamma_{k+1}$ where
$\eta_i = \sum_j \lambda_j - \lambda_i$.
Suppose $0\leq f_{1}\leq\cdots\leq f_{n}$ are the eigenvalues
of $\{F^{i\bar j}(r)\}$ of (\ref{eqn-quotient-2}).
If $1\leq k\leq n-1$, we have
\[
f_{1}\geq c_{n,k,l} \sum_{i=1}^{n}f_{i}
\]
for some uniform positive constant $c_{n,k}$ only depending on $n$, $k$ and $l$.
\end{lem}

\begin{proof}
We only need to prove the lemma when the matrix $r$ is diagonal. We assume $\lambda_{1}(r)\geq\cdots\geq\lambda_{n}(r)$.
Hence, $\eta_{1}\leq\cdots\leq\eta_{n}$. We compute that
\[
f_{i}=\frac{1}{k-l}\Big(\frac{\sigma_{k}}{\sigma_{l}}\Big)^{\frac{1}{k-l}-1}\sum_{p\neq i}\frac{\sigma_{k-1}(\eta|p)\sigma_{l}-\sigma_{k}\sigma_{l-1}(\eta|p)}{\sigma_{l}^{2}}.
\]
Direct calculations show that
\[
\begin{aligned} & \sigma_{k-1}(\eta|p)\sigma_{l}-\sigma_{k}\sigma_{l-1}(\eta|p)\\
= &\; \sigma_{k-1}(\eta|p)\sigma_{l}(\eta|p)-\sigma_{k}(\eta|p)\sigma_{l-1}(\eta|p)\\
= &\; \sigma_{k-1}(\eta|p)\sigma_{l}(\eta|p)(1-\alpha_{p}),
\end{aligned}
\]
where $\alpha_{p}$ is defined as
\[
\alpha_{p}:=\frac{\sigma_{k}(\eta|p)}{\sigma_{k-1}(\eta|p)}\frac{\sigma_{l-1}(\eta|p)}{\sigma_{l}(\eta|p)}.
\]
So we have
\[
f_{i}=\frac{1}{k-l}\Big(\frac{\sigma_{k}}{\sigma_{l}}\Big)^{\frac{1}{k-l}-1}\sum_{p\neq i}\frac{\sigma_{k-1}(\eta|p)\sigma_{l}(\eta|p)(1-\alpha_{p})}{\sigma_{l}^{2}}.
\]
Since $\eta \in \Gamma_{k+1}$, we see $\sigma_k (\eta | p) >0$ for every $1\leq p \leq n$.
By Newton-MacLaurin
inequality (Lemma \ref{NewMac} with $l = s = 0$), we have, for $1\leq p\leq n$,
\begin{equation}
0<\alpha_{p}\leq\frac{l(n-k)}{k(n-l)}<1. \label{eq:number fact}
\end{equation}
Recall that $\eta_1 \leq \cdots \leq \eta_n$.
By Lemma \ref{LT}, for any $q$ such that $1\leq q \leq n-1$ and $\eta \in \Gamma_q$,
we have $\sigma_{q-1}(\eta|2)\geq c_{n,q}\sum_i \sigma_{q-1}(\eta|i)$.
Especially, for $1\leq l<k\leq n-1$ and $\eta \subset \Gamma_{k+1} \subset \Gamma_k \subset \Gamma_{l+1}$, we have
\[
\sigma_{k-1}(\eta|2)\geq c_{n,k}\sum_{i=1}^{n}\sigma_{k-1}(\eta|i)\;\;\mbox{and}\;\;
\sigma_{l}(\eta|2)\geq c_{n,l}\sum_{i=1}^n \sigma_{l}(\eta|i).
\]
Now we can estimate
\[
\begin{aligned}f_{1}\geq & \;\frac{1}{k-l}\Big(\frac{\sigma_{k}}{\sigma_{l}}\Big)^{\frac{1}{k-l}-1}\frac{\sigma_{k-1}(\eta|2)\sigma_{l}(\eta|2)(1-\alpha_{2})}{\sigma_{l}^{2}}\\
\geq & \;\frac{1}{k-l}\Big(\frac{\sigma_{k}}{\sigma_{l}}\Big)^{\frac{1}{k-l}-1}\frac{c_{n,k}\sum_{i=1}^{n}\sigma_{k-1}(\eta|i)c_{n,l}\sum\sigma_{l}(\eta|i)}{\sigma_{l}^{2}}(1-\alpha_{2})\\
\geq & \;\frac{1}{k-l}\Big(\frac{\sigma_{k}}{\sigma_{l}}\Big)^{\frac{1}{k-l}-1}\frac{c_{n,k,l}\sum_{i=1}^{n}\sigma_{k-1}(\eta|i)\sigma_{l}(\eta|i)}{\sigma_{l}^{2}}(1-\alpha_{2})\\
\geq & \;\frac{1}{k-l}\Big(\frac{\sigma_{k}}{\sigma_{l}}\Big)^{\frac{1}{k-l}-1}\frac{c_{n,k,l}\sum_{i=1}^{n}\sigma_{k-1}(\eta|i)\sigma_{l}(\eta|i)(1-\alpha_{i})}{\sigma_{l}^{2}}\\
= & \;c_{n,k,l}\frac{1}{k-l}\Big(\frac{\sigma_{k}}{\sigma_{l}}\Big)^{\frac{1}{k-l}-1}\sum_{i=1}^{n}\frac{\sigma_{k-1}(\eta|i)\sigma_{l}-\sigma_{k}\sigma_{l-1}(\eta|i)}{\sigma_{l}^{2}},
\end{aligned}
\]
where the fourth inequality is due to (\ref{eq:number fact}) and $c_{n,k,l}$ may be different. 

Note that
\[
\begin{aligned}\sum_{i=1}^{n}f_{i}= & \;\frac{1}{k-l}\Big(\frac{\sigma_{k}}{\sigma_{l}}\Big)^{\frac{1}{k-l}-1}\sum_{i=1}^{n}\sum_{p\neq i}\frac{\sigma_{k-1}(\eta|p)\sigma_{l}-\sigma_{k}\sigma_{l-1}(\eta|p)}{\sigma_{l}^{2}}\\
= & \;\frac{1}{k-l}\Big(\frac{\sigma_{k}}{\sigma_{l}}\Big)^{\frac{1}{k-l}-1}(n-1)\sum_{i=1}^{n}\frac{\sigma_{k-1}(\eta|i)\sigma_{l}-\sigma_{k}\sigma_{l-1}(\eta|i)}{\sigma_{l}^{2}}.
\end{aligned}
\]
Therefore, we obtain
$f_{1}\geq\frac{c_{n,k,l}}{n-1}\sum_{i=1}^{n}f_{i}$.
\end{proof}

\begin{rmk}
By Newton-MacLaurin inequality (see Lemma \ref{NewMac} with $l = s = 0$ and $r=k-1$ ), we see that
\begin{equation}
\begin{aligned}
\sum \sigma_{k-1}(\eta|i) = (n-k+1) \sigma_{k-1} (\eta) \geq (n-k+1) C_n^{k-1}  \Big(\frac{\sigma_k}{C_n^k} \Big)^{\frac{k-1}{k}}.
\end{aligned}
\end{equation}
For equation \eqref{eqn'}, we obtain
\begin{equation}
\sum f_i (\lambda) \geq \frac{n-1}{k} (n-k+1) \frac{C_n^{k-1}}{(C_n^k)^{\frac{k-1}{k}}}
\end{equation}
for $\lambda \in \mathbb{R}^n$ such that $\eta\in \Gamma_k$.
Hence, by Lemma \ref{key}, equation \eqref{eqn'} is strictly elliptic.
Similar result can be shown for quotient type equation \eqref{eqn-quotient-2} with $\eta \in \Gamma_{k+1}$ by Lemma \ref{NewMac} and Lemma \ref{key-2}.

\end{rmk}
In the sections below, we will use $\sum_{i=1}^n f_i\ge c_0$ and $f_i\ge c_0\sum_{j=1}^n f_j$ for convenience,
where $c_0$ is a positive constant depending on $n$, $k$ and $l$ if $l$ exists.

\section{Gradient estimates}

In this section, we prove the $C^{1}$ estimates. We always assume
that the conditions in Theorem \ref{thm} hold. We first show the
following interior gradient estimate.
\begin{thm}
\label{thm:interior gradient}
Suppose $u\in C^3(\Omega)$ is a solution to \eqref{eqn-3} or \eqref{eqn-6}. 
Assume $0\in\Omega$ and $B_{r}(0)\subset\Omega$.
Then, we have
\[
|\nabla u|(0)\leq\frac{C}{r},
\]
where $C$ depends on $|u|_{C^{0}}$ and other known data.
\end{thm}

\begin{proof}
Consider the following auxiliary function on
$B_{r}(0)\subset\Omega$,
\[
G(z)=\log|\nabla u|+h(u)+\log\zeta(z),
\]
where $\zeta(z)=r^{2}-|z|^{2}$ and $|\nabla u|^2=\sum_{i=1}^nu_ku_{\bar{k}}$. Assume $G$ attains its maximum at
$z_{0}\in B_{r}(0)$. At $z_{0}$, we have
\begin{equation}\label{Gradient- first derivative}
0=G_{i}=\frac{|\nabla u|_{i}^{2}}{2|\nabla u|^{2}}+h'u_{i}+\frac{\zeta_{i}}{\zeta}
\end{equation}
and
\[
\begin{aligned}0\geq & F^{i\bar{j}}G_{i\bar{j}}\\
= &\; F^{i\bar{j}}\frac{|\nabla u|_{i\bar{j}}^{2}}{2|\nabla u|^{2}}-F^{i\bar{j}}\frac{|\nabla u|_{i}^{2}|\nabla u|_{\bar{j}}^{2}}{2|\nabla u|^{4}}\\
& + h' F^{i\bar{j}}u_{i\bar{j}} + h''F^{i\bar{j}}u_{i}u_{\bar{j}}
+F^{i\bar{j}}\frac{\zeta_{i\bar{j}}}{\zeta}-F^{i\bar{j}}\frac{\zeta_{i}\zeta_{\bar{j}}}{\zeta^{2}}.
\end{aligned}
\]

It is immediate to see that
\[
F^{i\bar{j}}\frac{\zeta_{i\bar{j}}}{\zeta}-F^{i\bar{j}}\frac{\zeta_{i}\zeta_{\bar{j}}}{\zeta^{2}}=-\frac{\sum F^{i\bar{i}}}{\zeta}-\frac{F^{i\bar{j}}z_{\bar{i}}z_{j}}{\zeta^{2}}.
\]
By Cauchy-Schwartz inequality and (\ref{Gradient- first derivative}), we obtain
\[
F^{i\bar{j}}\frac{|\nabla u|_{i}^{2}|\nabla u|_{\bar{j}}^{2}}{2|\nabla u|^{4}}\leq4(h')^{2}F^{i\bar{j}}u_{i}u_{\bar{j}}+\frac{4}{\zeta^{2}}F^{i\bar{j}}\zeta_{i}\zeta_{\bar{j}}.
\]
Direct calculation shows
\[
F^{i\bar{j}}|\nabla u|_{i\bar{j}}^{2}
= F^{i\bar{j}}(u_{ki}u_{\bar{k}\bar{j}}+u_{\bar{k}i}u_{k\bar{j}})+u_{k} \tilde{\psi}_{\bar{k}} + u_{\bar{k}} \tilde{\psi}_{k}.
\]
We choose $h=\delta(u+C_{0})^{2}$ where $|u|_{C^{0}}\leq C_{0}-1$
and $\delta>0$. We have
\[
h'=2\delta(u+C_{0})\geq2\delta
\]
and
\[
h''-4(h')^{2}=2\delta-16\delta^{2}(u+C_{0})^{2}\geq\delta
\]
for sufficiently small $\delta$. Recall that $F^{i\bar{j}}u_{i\bar{j}}\geq0$, see \eqref{eq:fundamental equality 1}.
Hence, by Lemma \ref{key}, we have
\[
\begin{aligned}0\geq & - C +\delta F^{i\bar{j}}u_{i}u_{\bar{j}}-\frac{1}{\zeta}\sum_{i=1}^{n}F^{i\bar{i}}-\frac{5r^{2}}{\zeta^{2}}\sum_{i=1}^{n}F^{i\bar{i}}\\
\geq & - C +\delta c_{n,k}|\nabla u|^{2}\sum_{i=1}^{n}F^{i\bar{i}}-\frac{1}{\zeta}\sum_{i=1}^{n}F^{i\bar{i}}-\frac{5r^{2}}{\zeta^{2}}\sum_{i=1}^{n}F^{i\bar{i}}.
\end{aligned}
\]

Assume $|\nabla u(z_{0})|^{2}\geq\frac{C}{\frac{\delta}{2}c_{n,k}c_{0}}$.
We arrive
at
\[
0\geq(\frac{\delta c_{n,k}}{2}|\nabla u|^{2}-\frac{6r^{2}}{\zeta^{2}})\sum_{i=1}^{n}F^{i\bar{i}}
\]
from which we can derive that
\[
\zeta^{2}(z_{0})|\nabla u|^{2}(z_{0})\leq\frac{Cr^{2}}{\delta c_{n,k}}.
\]
Therefore, by $G(0)\leq G(z_{0})$, we obtain
\[
|\nabla u|(0)\leq\frac{C}{r}.
\]
\end{proof}
Now we prove the global gradient estimates by the following theorem.
\begin{thm}\label{global gradient}
Suppose $u$ is a $C^{3}$ $k$-admissible solution to \eqref{eqn-3} or \eqref{eqn-6}.
Then, we have
\[
\sup_{\Omega} |\nabla u|\le C,
\]
where $C$ depends on $|u|_{C^0}$ and other known data.
\end{thm}

\begin{proof}
Denote $w=u+\varphi(z,u)d$ and $u_\nu=\varphi(z,u)=-\beta u+\varphi(z)$ on $\partial \Omega$.
Consider the following function
\[
G(z)=\log|\nabla w |+g(d)+h(u),
\]
where $d$ is a smooth function and near boundary equals to the distance
function to the boundary, $g= Ad$ for a large positive constant $A$, and $h$ is a smooth function to be determined later.
Suppose that $G$ attains its maximum at $z_{0}\in\bar{\Omega}$,
i.e. $\max_{\bar{\Omega}}G(z) = G(z_{0})$.
We divide the proof into three cases.

Case 0: $z_{0}$ is in $\Omega_{\mu}=\{x | d(z,\partial\Omega)\geq\mu\}$.
We can bound $|\nabla u|(z_{0})$ by Theorem \ref{thm:interior gradient},
i.e. we have $|\nabla u|(z_{0})\leq\max_{\Omega_{\mu}}|\nabla u|\leq C$.

Case 1: $z_{0}$ is on the boundary $\partial\Omega$. Notice that
$w_{\nu}=u_{\nu}+\varphi_{\nu}d+\varphi d_{\nu}=0$ on $\partial \Omega$, where $\nu$ is the unit
outer normal vector. Hence, at $z_0$,
\begin{align*}
0\le\frac{\partial}{\partial \nu} G & =\frac{\frac{1}{4}|D w|_{\nu}^{2}}{|\nabla w|^{2}}+g'd_{\nu}+h'u_{\nu}\\
 & =\frac{  \frac{1}{2}D_kwD_{k\nu}w}{|\nabla w|^{2}}+g'd_{\nu}+h'u_{\nu}\\
 & \le 2\sup_{\partial \Omega}\{|\Pi_{ij}|\}- A +h'\varphi( z_0 ,u),
\end{align*}
which yields a contradiction to the larger choice of $$A=2\sup_{\partial \Omega}\{|\Pi_{ij}|\}+\sup_{\Omega}|h'||\varphi|+1.$$

Case 2: $z_{0}$ is in $\Omega\backslash\Omega_{\mu}$.
Differentiate $G $ at $z_{0}$ once to obtain that
\[
0 = G_{i}(z_{0})=\frac{|\nabla w|_{i}^{2}}{|\nabla w|^{2}}+g'd_{i}+h'u_{i},
\]
and a second time to get that
\begin{align}
0 & \ge F^{i\bar{j}} G_{i\bar{j}}\label{eq:C1 main-1}\\
 & =F^{i\bar{j}}\frac{|\nabla w|_{i\bar{j}}^{2}}{|\nabla w|^{2}}-F^{i\bar{j}}\frac{|\nabla w|_{i}^{2}|\nabla w|_{\bar{j}}^{2}}{|\nabla w|^{4}} + g'F^{i\bar{j}}d_{i\bar{j}}+h'F^{i\bar{j}}u_{i\bar{j}}+h''F^{i\bar{j}}u_{i}u_{\bar{j}}\nonumber \\
 & =F^{i\bar{j}}\frac{|\nabla w|_{i\bar{j}}^{2}}{|\nabla w|^{2}}-F^{i\bar{j}}(g'd_{i}+h'u_{i})(g'd_{\bar{j}}+h'u_{\bar{j}}) + g'F^{i\bar{j}}d_{i\bar{j}}+h'F^{i\bar{j}}u_{i\bar{j}}+h''F^{i\bar{j}}u_{i}u_{\bar{j}}\nonumber \\
 & \ge F^{i\bar{j}}\frac{|\nabla w|_{i\bar{j}}^{2}}{|\nabla w|^{2}}-2(g')^{2}F^{i\bar{j}}d_{i}d_{\bar{j}}+(h''-2(h')^{2})F^{i\bar{j}}u_{i}u_{\bar{j}}+g'F^{i\bar{j}}d_{i\bar{j}}+h'F^{i\bar{j}}u_{i\bar{j}},\nonumber
\end{align}
where in the last inequality we used Cauchy-Schwarz inequality.

Now let us deal with $F^{i\bar{j}}\frac{|\nabla w|_{i\bar{j}}^{2}}{|\nabla w|^{2}}$.
First, we compute that
\begin{align}
F^{i\bar{j}}|\nabla w|_{i\bar{j}}^{2} & =F^{i\bar{j}}(w_{k}w_{\bar{k}i\bar{j}}+w_{ki}w_{\bar{k}\bar{j}}+w_{ki\bar{j}}w_{\bar{k}}+w_{k\bar{j}}w_{\bar{k}i})\label{eq:third derivative-1}\\
 & =w_{k}F^{i\bar{j}}w_{\bar{k}i\bar{j}}+F^{i\bar{j}}w_{ki\bar{j}}w_{\bar{k}}+F^{i\bar{j}}(w_{ki}w_{\bar{k}\bar{j}}+w_{k\bar{j}}w_{\bar{k}i}).\nonumber
\end{align}
Recall that $\varphi(z,u)= - \beta u + \phi(z)$. We have
\begin{align*}
 F^{i\bar{j}}w_{\bar{k}i\bar{j}}
 = &\, F^{i\bar{j}}u_{\bar{k}i\bar{j}}+F^{i\bar{j}}(\varphi d)_{\bar{k}i\bar{j}}\\
 = &\, \tilde{\psi}_{\bar{k}}+F^{i\bar{j}}[(-\beta u+\phi(z))d]_{\bar{k}i\bar{j}}\\
 = &\, F^{i\bar{j}}(-\beta u_{\bar{k}i}d_{\bar{j}}-\beta u_{\bar{k}\bar{j}}d_{i}-\beta u_{i\bar{j}}d_{\bar{k}}
 -\beta u_{\bar{k}}d_{i\bar{j}}-\beta u_{i}d_{\bar{k}\bar{j}}\\
 & -\beta u_{\bar{j}}d_{\bar{k}i}-\beta ud_{\bar{k}i\bar{j}})
  +F^{i\bar{j}}(\phi(x)d)_{\bar{k}i\bar{j}} + \tilde{\psi}_{\bar{k}}(1-\beta d),
\end{align*}
and then by Cauchy-Schwarz inequality we get
\begin{equation}
\label{eq:third derivative 1-1}
\begin{aligned}
 & F^{i\bar{j}}w_{ki\bar{j}}w_{\bar{k}}+w_{k}F^{i\bar{j}}w_{\bar{k}i\bar{j}} \\
\ge &\; (\tilde{\psi}_{\bar{k}}w_{k} + \tilde{\psi}_{k}w_{\bar{k}})(1-\beta d)-\varepsilon F^{i\bar{j}}(u_{ki}u_{\bar{k}\bar{j}}+u_{k\bar{j}}u_{\bar{k}i}) \\
&-C_{\varepsilon}\sum_{i=1}^{n}F^{i\bar{i}}(|\nabla u|^{2}+|\nabla u|).
\end{aligned}
\end{equation}
Direct calculation shows that
\begin{equation}
\label{eq:third derivative 2-1}
\begin{aligned}
 & F^{i\bar{j}}(w_{ki}w_{\bar{k}\bar{j}}+w_{k\bar{j}}w_{\bar{k}i}) \\
 = &\; F^{i\bar{j}}( u(1-\beta d)+\phi(x)d  )_{ki}(u(1-\beta d)+\phi(x)d)_{\bar{k}\bar{j}}\\
  & +  F^{i\bar{j}}(u(1-\beta d)+\phi(x)d)_{\bar{k}i}(u(1-\beta d)+\phi(x)d)_{k\bar{j}} \\
 \ge& (1-\beta d)^{2}F^{i\bar{j}}(u_{ki}u_{\bar{k}\bar{j}}+u_{k\bar{j}}u_{\bar{k}i})(1-\varepsilon)
-C_{\varepsilon}\sum_{i=1}^{n}F^{i\bar{i}}(|\nabla u|^{2}+|\nabla u|+1),
\end{aligned}
\end{equation}
where in the last inequality we used Cauchy-Schwarz inequality. By
(\ref{eq:third derivative-1}), (\ref{eq:third derivative 1-1}) and
(\ref{eq:third derivative 2-1}), for $\mu$ chosen sufficiently small,
we obtain
\[
F^{i\bar{j}}|\nabla w|_{i\bar{j}}^{2}\ge - C |\nabla u|^2 -C\sum_{i=1}^{n}F^{i\bar{i}}(|\nabla u|^{2}+|\nabla u|+1).
\]

From (\ref{eq:C1 main-1}), we have
\begin{equation}
\label{eq:C1 main-2}
\begin{aligned}
0  
  \ge &\; -C -C\sum_{i=1}^{n}F^{i\bar{i}} -2(g')^{2} F^{i\bar{j}}d_{i}d_{\bar{j}} \\
 & + (h''-2(h')^{2})F^{i\bar{j}}u_{i}u_{\bar{j}}+g'F^{i\bar{j}}d_{i\bar{j}} + h'F^{i\bar{j}}u_{i\bar{j}}.
\end{aligned}
\end{equation}
We choose $h=\delta(u+C_{0})^{2}$ where $|u|_{C^{0}}\leq C_{0}-1$
and $\delta>0$.
We have
\[
h' > 2\delta  \;\; \mbox{and} \;\; h'' - 2 (h')^2 > \delta
\]
for $\delta > 0$ sufficiently small.
By the fact that $F^{i\bar{j}}u_{i}u_{\bar{j}}\ge c_{0}\sum_{i=1}^{n}F^{i\bar{i}}|\nabla u|^{2}$,
$\sum_{i=1}^n F^{i\bar i} \geq c_0$,
 and $F^{i\bar j} u_{i\bar j} \geq 0$, \eqref{eq:C1 main-2} yields that
\[
|\nabla u|^{2}\le C.
\]
\end{proof}

\begin{rmk}
One can show that the global $C^1$ estimate still holds for $k$-admissible solutions to \eqref{eqn-3} with the right hand term $\psi (z,u,Du)$
by the proof with minor changes.
\end{rmk}

\section{Second order estimates}

In this section, we prove the second order estimates. We first reduce
the second order estimates to the double normal derivative on the
boundary by the following theorem.
\begin{thm}
\label{thm:second derivative boundary reduction}Suppose $u$ is a
$C^{4}$ $k$-admissible solution to \eqref{eqn-3} or \eqref{eqn-6}. Then, we have
\[
\sup_{ (z, \xi) \in \Omega\times S^{2n-1}} D_{\xi\xi} u (z) \le C(1+\sup_{\partial\Omega}|D_{\nu\nu}u|),
\]
where $C$ depends on $|u|_{C^1}$ and other known data.
\end{thm}

\begin{proof}
Define  $h=e^{-Ar}$, where $r$ is a
 function in $C^{2}(\bar{\Omega})$ with  $r|_{\partial \Omega} = 0$  and $D_\nu r=1$ on $\partial\Omega$,
$A=1+2\max_{\partial\Omega}\{|\Pi_{ij}|\}+|\beta|$ is a constant,
and  $\Pi_{ij}$   is
the second fundamental form of the boundary.
We adopt the following auxiliary function
\[
\Phi(z,\zeta)=h(r)(D_{\zeta\zeta}u-v(z,\zeta))+|\nabla u|^{2}
\]
where $v(x,\zeta) = 2\langle\zeta, \nu\rangle \langle\zeta', D\phi - \beta D u - D_l u D\nu^l\rangle \equiv  a^{l}D_{l}u+b$,  $\zeta'=\zeta-\langle\zeta, \nu\rangle \nu$, $ a^l = - 2 \langle\zeta, \nu\rangle \langle\zeta', D \nu^l\rangle- 2\beta \langle\zeta , \nu\rangle (\zeta')^l$, $b = 2 \langle\zeta, \nu \rangle \langle\zeta', D \phi\rangle$. For any given $\zeta\in S^{2n-1}$,
suppose $\max_{z\in \bar \Omega}\Phi(z,\zeta)$ is attained at $z_0 \in \bar \Omega$.

$\clubsuit$ Case 1: $z_{0}\in\Omega$. We prove that this case
will not happen with proper coefficients. Differentiating $\Phi$
at $z_{0}$, we obtain
\[
0=\Phi_{i}=h'r_{i}(D_{\zeta\zeta}u-v(z,\zeta))+h(r)(D_{\zeta\zeta}u-v(z,\zeta))_{i}+u_{k}u_{\bar{k}i}+u_{\bar{k}}u_{ki},
\]
which yields that
\begin{equation}
(D_{\zeta\zeta}u-v(z,\zeta))_{i}=-\frac{h'r_{i}(D_{\zeta\zeta}u-v(z,\zeta))+u_{k}u_{\bar{k}i}+u_{\bar{k}}u_{ki}}{h}.\label{eq:second reduction estimate -first derivative}
\end{equation}
Differentiating $\Phi$ a second time to get
\begin{align*}
  F^{i\bar{j}}\Phi_{i\bar{j}}
\ge &\; h'F^{i\bar{j}}r_{i\bar{j}}(D_{\zeta\zeta}u-v(z,\zeta))\\
 & +F^{i\bar{j}}h''r_{i}r_{\bar{j}}(D_{\zeta\zeta}u-v(z,\zeta))+h'r_{i}F^{i\bar{j}}(D_{\zeta\zeta}u-v(z,\zeta))_{\bar{j}}\\
 & +F^{i\bar{j}}h'r_{\bar{j}}(D_{\zeta\zeta}u-v(z,\zeta))_{i}+hF^{i\bar{j}}(D_{\zeta\zeta}u-v(z,\zeta))_{i\bar{j}}\\
 & +F^{i\bar{j}}(u_{k\bar{j}}u_{\bar{k}i}+u_{\bar{k}\bar{j}}u_{ki})+F^{i\bar{j}}(u_{k}u_{\bar{k}i\bar{j}}+u_{\bar{k}}u_{ki\bar{j}}),
\end{align*}
and then by (\ref{eq:second reduction estimate -first derivative}),
we obtain

\begin{align*}
  F^{i\bar{j}}\Phi_{i\bar{j}}
\ge &\; 2 h' F^{i\bar{j}} r_{i} \Big(-\frac{h'r_{\bar{j}}(D_{\zeta\zeta}u-v(z,\zeta)) 
+ u_{k}u_{\bar{k}\bar{j}}+u_{\bar{k}}u_{k\bar{j}}}{h} \Big)\\
 & + \big( h'F^{i\bar{j}}r_{i\bar{j}} + F^{i\bar{j}}h''r_{i}r_{\bar{j}} \big) (D_{\zeta\zeta}u-v(z,\zeta))
  + F^{i\bar{j}}(u_{k\bar{j}}u_{\bar{k}i} + u_{\bar{k}\bar{j}}u_{ki}) \\
 & +F^{i\bar{j}}(u_{k}u_{\bar{k}i\bar{j}}+u_{\bar{k}}u_{ki\bar{j}})
  +hF^{i\bar{j}}(D_{\zeta\zeta}u-v(z,\zeta))_{i\bar{j}}\\
\ge &\; F^{i\bar{j}}r_{i}r_{\bar{j}}(D_{\zeta\zeta}u-v(z,\zeta)) \Big( h''-2\frac{(h')^{2}}{h} \Big)
  - 2 \frac{h'}{h} F^{i\bar{j}} r_{i} (u_{k}u_{\bar{k}\bar{j}}+u_{\bar{k}}u_{k\bar{j}})\\
& 
  +h'F^{i\bar{j}}r_{i\bar{j}}(D_{\zeta\zeta}u-v(z,\zeta))
  +F^{i\bar{j}}(u_{k\bar{j}}u_{\bar{k}i}+u_{\bar{k}\bar{j}}u_{ki})
  + u_{k} \tilde{\psi}_{\bar{k}} + u_{\bar{k}} \tilde{\psi}_{k} \\
  & + h \tilde{\psi}_{\zeta\zeta} 
  -h F^{i\bar{j}}(a_{i\bar{j}}^{l}D_{l}u + 2 a_{i}^{l}(D_{l}u)_{\bar{j}}+a^{l}(D_{l}u)_{i\bar{j}}+b_{i\bar{j}}),
\end{align*}
where we used (\ref{eq:fundamental equality 2}), 
$G^{i\bar{j,}k\bar{l}}\eta_{i\bar{j}\xi}\eta_{k\bar{l}\xi}+G^{i\bar{j}}\eta_{i\bar{j}\xi\xi} = \tilde{\psi}_{\xi\xi}$
and $G^{i\bar{j,}k\bar{l}} \leq 0$ in the last inequality. By Cauchy-Schwarz inequality,
we see
\begin{align*}
 & -2 \frac{h'}{h}r_{i}F^{i\bar{j}}(u_{k}u_{\bar{k}\bar{j}}+u_{\bar{k}}u_{k\bar{j}})\\
\ge & -\varepsilon F^{i\bar{j}}(u_{\bar{k}\bar{j}}u_{ki}+u_{k\bar{j}}u_{\bar{k}i}) 
-\frac{2}{\varepsilon}|\nabla u|^{2}\frac{(h')^{2}}{h^{2}}F^{i\bar{j}}r_{i}r_{\bar{j}}.
\end{align*}
Similarly, we have
\[\begin{aligned}
& - hF^{i\bar{j}}(a_{i\bar{j}}^{l}D_{l}u+ 2 a_{i}^{l}(D_{l}u)_{\bar{j}}+a^{l}(D_{l}u)_{i\bar{j}}+b_{i\bar{j}}) \\
\ge & \; -\frac{1}{4}F^{i\bar{j}}(D_{l}u)_{\bar{j}}(D_{l}u)_{i} -C_1(h+h^{2}) \sum F^{i\bar{i}}
- ha^{l}D_{l} \tilde{\psi}.
\end{aligned}\]
Note that $\sum_{l=1}^{2n}D_{l\bar{j}}uD_{li}u=\sum_{p=1}^n2u_{p\bar{j}}u_{\bar{p}i}+2u_{\bar{p}\bar{j}}u_{pi}.$
Above all, we then arrive at
\[\begin{aligned}
F^{i\bar{j}}\Phi_{i\bar{j}}\ge &\; F^{i\bar{j}}r_{i}r_{\bar{j}}(D_{\zeta\zeta}u-v(z,\zeta))(h''-2\frac{(h')^{2}}{h})+h'F^{i\bar{j}}r_{i\bar{j}}(D_{\zeta\zeta}u-v(z,\zeta))\\
 & +(\frac{1}{2} - \varepsilon)F^{i\bar{j}}(u_{k\bar{j}}u_{\bar{k}i}+u_{\bar{k}\bar{j}}u_{ki}) -\frac{2}{\varepsilon}|\nabla u|^{2}\frac{(h')^{2}}{h^{2}}F^{i\bar{j}}r_{i}r_{\bar{j}} + h \tilde{\psi}_u u_{\zeta\zeta} \\
 &  
  - C_3h - C_2-C_1(h+h^{2}) \sum F^{i\bar{i}}.\\
\end{aligned}\]
By Lemma \ref{key} and taking $\varepsilon=1/8$, we obtain that
\begin{align*}
F^{i\bar{j}}\Phi_{i\bar{j}}
\ge & -C_4\times(D_{\zeta\zeta}u-v(z,\zeta)) \sum F^{i\bar{i}}
 - \frac{2}{\varepsilon} |\nabla u|^{2}\frac{(h')^{2}}{h^{2}} \sum F^{i\bar{i}}\\
 & -C_2-C_3h- C_1(h+h^{2}) \sum F^{i\bar{i}}
 + \frac{ c_0}{4} \sum_{j=1}^{n}F^{j\bar{j}}\sum_{k,i}(u_{k\bar{i}}u_{\bar{k}i}+u_{\bar{k}\bar{i}}u_{ki})\\
\ge & -C_{4}D_{\zeta\zeta}u \sum F^{i\bar{i}} - \frac{2}{\varepsilon} |\nabla u|^{2}\frac{(h')^{2}}{h^{2}} \sum F^{i\bar{i}}
-C_5\sum F^{i\bar{i}}\\
 & -C_{2}-C_{3}h - C_1(h+h^{2}) \sum F^{i\bar{i}} +\frac{ c_0 }{4} \sum_{i=1}^{n}F^{i\bar{i}} ( D_{\zeta\zeta} u )^{2},
\end{align*}
which is positive as long as $D_{\zeta\zeta} u$ is large enough.
For example, it is positive when $D_{\zeta\zeta} u>D$. Here,
\[
D:=2\sqrt{\frac{C^2_{4}}{c^2_{0}}+\frac{C_6}{c_0}}+2\frac{C_{4}}{c_0},
\]
where $C_6=16B^2A^2+(C_1+\frac{C_3}{c_0})e^{A\max_{\Omega}|r|}+C_1e^{2A\max_{\Omega}|r|}+C_5+\frac{C_2}{c_0}$;$B=\sup_{\Omega}|\nabla u|$;  $C_3$ depends on $A$, $\max_{\Omega}|r|$ and $\sup_{\Omega}(|\nabla^2r|+|\nabla r|)$; $C_4$ depends on $\sup_{\Omega}|\nabla u|,|\nabla \phi|$ and $\partial \Omega $; $C_5$ depends on $\sup_{\Omega}|\nabla \phi|,
\beta$ and $\partial \Omega$.

Denote $\max_{(z,\zeta)\in \bar \Omega\mathbb{\times S}^{2n-1}}\Phi(z,\zeta)=\Phi(z_{0},\zeta_{0})$.
If $D_{\zeta_0\zeta_0} u(z_0)>D$, by the analysis in Case 1, then we know $z_0\in \partial \Omega$. Otherwise, we have proved this theorem.
Now let us deal with Case
2.

$\clubsuit$ Case 2: $z_{0}\in\partial\Omega.$ We further divide
this case into two subcases according to whether the direction $\zeta_{0}$ is tangential or non-tangential to the boundary.

(a) If $\zeta_{0}$ is non-tangential at $z_{0}\in\partial\Omega$, then
we can write $\zeta_{0}=\beta_{1}\tau+\beta_{2}\nu$, where $\tau\in\mathbb{S}^{2n-1}$
is tangential at $z_{0}$, that is $\langle\tau,\nu\rangle=0$, $\beta_{1}=\langle\zeta_{0},\tau\rangle$,
$\beta_{2}=\langle\zeta_{0},\nu\rangle\ne0$, and $\beta_{1}^{2}+\beta_{2}^{2}=1$.
Then, we have
\begin{align*}
D_{\zeta_{0}\zeta_{0}}u(z_{0})= &\; \beta_{1}^{2}D_{\tau\tau}u(z_{0})+\beta_{2}^{2}D_{\nu\nu}u(z_{0})+2\beta_{1}\beta_{2}D_{\tau\nu}u(z_{0})\\
= &\; \beta_{1}^{2}D_{\tau\tau}u(z_{0})+\beta_{2}^{2}D_{\nu\nu}u(z_{0})\\
 & +2(\xi_{0}\cdot\nu)[\xi_{0}-(\xi_{0}\cdot\nu)\nu] \cdot [D\phi-\beta Du-D_{l}uD\nu^{l}]
\end{align*}
from which we see
\[
\Phi(z_{0},\zeta_{0})= \; \beta_{1}^{2}\Phi(z_{0},\tau)+\beta_{2}^{2}\Phi(z_{0},\nu).
\]
By the definition of $\Phi(z_{0},\zeta_{0})$, we know
\[
\Phi(z_{0},\zeta_{0})\le \; \Phi(z_{0},\nu)\leq C_{7}(1+\max_{\partial\Omega}|D_{\nu\nu}u|).
\]

(b) If $\zeta_{0}$ is tangential at $z_{0}\in\partial\Omega$, then by \eqref{grad-equality} we
have
\begin{align*}
0\leq &\; D_{\nu}\Phi(z_{0},\zeta_{0})\\
= &\; -A(D_{\zeta_{0}\zeta_{0}}u-a^{l}D_{l}u-b)+D_{\nu}D_{\zeta_{0}\zeta_{0}}u\\
 & -D_{\nu}a^{l}D_{l}u-a^{l}D_{\nu}D_{l}u-D_{\nu}b + \frac{1}{2} D_{ k} u D_{\nu} D_{k} u\\
\le &\; -AD_{\zeta_{0}\zeta_{0}}u+D_{\nu}D_{\zeta_{0}\zeta_{0}}u + [ \frac{1}{2} D_{k}u-a^{k}]D_{\nu}D_{k}u+C_{8}.
\end{align*}
By the boundary condition, we know
\begin{align}
D_{\nu}D_{\zeta_{0}\zeta_{0}}u= & \; D_{\zeta_{0}\zeta_{0}}D_{\nu}u-(D_{\zeta_{0}\zeta_{0}}\nu^{k})D_{k}u-2(D_{\zeta_{0}}\nu^{k})D_{\zeta_{0}}D_{k}u\nonumber \\
= &\; D_{\zeta_{0}\zeta_{0}}(-\beta u+ \phi )-(D_{\zeta_{0}\zeta_{0}}\nu^{k})D_{k}u-2(D_{\zeta_{0}}\nu^{k})D_{\zeta_{0}}D_{k}u\nonumber \\
\leq &\; D_{\zeta_{0}\zeta_{0}}(-\beta u)+C_{9}-2(D_{\zeta_{0}}\nu^{k})D_{\zeta_{0}}D_{k}u.
\end{align}
By the same argument as Lemma 4.3 in \cite{LiSY}, we know
\[
D_{\nu}D_{k}u\le C_{10}(1+|D_{\nu\nu}u|)
\]
on $\partial\Omega$. Note that the proof of Lemma 4.3 is only related
to the Neumann boundary condition. Also, by the similar argument in
\cite{MQ}, we see
\[
-2(D_{\zeta_{0}}\nu^{k})D_{\zeta_{0}}D_{k}u\le C_{11}D_{\zeta_{0}\zeta_{0}}u,
\]
where $C_{11}=2\max_{\partial\Omega}\{|\Pi_{ij}|\}$ and $\Pi_{ij}$ is
the second fundamental form of the boundary. Therefore, we have
\[
0\leq( - A + C_{11} - \beta )D_{\zeta_{0}\zeta_{0}}u+C_{12}(1+|D_{\nu\nu}u|)+C_{13}.
\]
Taking $A=1+2\max_{\partial\Omega}\{|\Pi_{ij}|\}+|\beta|>C_{11}-\beta+1$, we get
\[
\Phi (z_{0},\zeta_{0})\leq C_{12}(1+\max_{\partial\Omega}|D_{\nu\nu}u|)+C_{13} +C_{14},
\]
where $C_{14}$ depends on $|u|_{C^1(\Omega)}$, $\max_{\Omega}|r|$ and $A$.
\end{proof}

\begin{rmk}
For general Neumann boundary data $\varphi(z,u)$, Theorem \ref{thm:second derivative boundary reduction} still holds by replacing $v(z,\zeta)$ in the above proof with
$v(z,\zeta) =2\langle\zeta, \nu\rangle \langle\zeta', D\varphi - D_l u D\nu^l\rangle $.
\end{rmk}

\begin{rmk}
For $F(u_{i\bar{j}})=\tilde{\psi}(z,u,Du)$ with $\tilde{\psi}(z,q,p)\in C^2(\Omega\times \mathbb{R}\times\mathbb{R}^n)$, Theorem \ref{thm:second derivative boundary reduction} still holds.  The main changes are to deal with $h\tilde{\psi}_{\zeta\zeta}$. By choosing $h=e^{-A_1r-A_2}$ for large $A_2$ such that $|h\tilde{\psi}_{p_kp_l}|<\frac{c_0}{8}$ and  (\ref{eq:second reduction estimate -first derivative}), we can also control $h\tilde{\psi}_{\zeta\zeta}$ and obtain Theorem \ref{thm:second derivative boundary reduction}.\end{rmk}

Now we estimate the double normal derivative on the boundary.
\begin{thm}\label{thm:double normal}

Suppose
 $u\in C^{4}(\Omega)\cap C^{3}(\overline{\Omega})$ is a $k$-admissible solution
to equation (\ref{eqn-3}) or \eqref{eqn-6}. Then, we have
\[
\max_{\partial\Omega}|D_{\nu\nu}u|\leq C ,
\]
where $ C $ depends on $n$, $k$, $\Omega$,
$|u|_{C^{1}}$, $|\psi|_{C^{1}}$ and $|\varphi|_{C^{3}}$.
\end{thm}

\begin{proof}
We adopt the idea in \cite{WangJ} for barrier function with minor
changes to suit our needs. Denote $M=\sup_{\partial\Omega}|u_{\nu\nu}|$.
As $\frac{\partial u}{\partial\nu}=\varphi(z,u)$, we construct the
following auxiliary function
\[
\Phi =\; <Du,Dr> - \varphi(z,u) + M^{-\frac{1}{2}}(<Du,Dr>-\varphi(z,u))^{2}+\frac{1}{2}Mr
\]
where $r$ is a defining function such that $r<0$ in $\Omega$, $r=0$ on $\partial \Omega$,
and $\frac{\partial r}{\partial\nu}=1$ on $\partial\Omega$. Define
\[
\Omega_{\mu}:=\{z\in\Omega:d(z,\partial\Omega)<\mu\}.
\]
On $\partial\Omega$, it is obvious that $\Phi=0$.
Take a small positive constant $\mu$ such that  $r=-d$ on $\Omega_{\mu}$,
where $d$ is the distance function to the boundary $\partial\Omega$.
On $\partial\Omega_{\mu}\backslash\partial\Omega$, we see $\Phi<C_{1}-C_{2}M^{-1/2}-\frac{1}{2}M\mu<0$,
when $M$ is large enough. Suppose
\[
\Phi(z_{0})=\max_{ \overline{ \Omega_{\mu} }}\Phi.
\]

If $\Phi$ achieves its maximum in $\Omega_{\mu}$, i.e. $z_0 \in \Omega_\mu$, we then have
\begin{align*}
0= \; \Phi_{i}(z_{0})= &\; [<Du,Dr>-\varphi(z,u)]_{i}\left(1+2M^{-\frac{1}{2}}(<Du,Dr>-\varphi(z,u))\right)\\
 & +\frac{1}{2}Mr_{i},
\end{align*}
and
\begin{align*}
0\ge &\; \Phi_{i\bar{j}}(z_0)\\
= &\; [<Du,Dr>-\varphi(z,u)]_{i\bar{j}}\left((1+2M^{-\frac{1}{2}}(<Du,Dr>-\varphi(z,u))\right)\\
 & +2M^{-\frac{1}{2}}[<Du,Dr>-\varphi(x,u)]_{i} [ <Du,Dr>-\varphi(z,u) ]_{\bar{j}}\\
 & +\frac{1}{2}Mr_{i\bar{j}}\\
= &\; [<Du,Dr>-\varphi(z,u)]_{i\bar{j}}\left((1+2M^{-\frac{1}{2}}(<Du,Dr>-\varphi(z,u))\right)\\
 & +\frac{2M^{3/2}r_{i}r_{\bar{j}}}{4\left(1+2M^{-\frac{1}{2}}(<Du,Dr>-\varphi(z,u))\right)^{2}}
 + \frac{1}{2}Mr_{i\bar{j}}.
\end{align*}

Note that
\begin{align*}
& F^{i\bar{j}}[  <Du,Dr>-\varphi(z,u)]_{i\bar{j}}\\
= &\; F^{i\bar{j}}[(D_{\alpha}u)_{i\bar{j}}D_{\alpha}r+(D_{\alpha}u)_{i}(D_{\alpha}r)_{\bar{j}}+(D_{\alpha}u)_{\bar{j}}(D_{\alpha}r)_{i}+D_{\alpha}u(D_{\alpha}r)_{i\bar{j}}]\\
 & -F^{i\bar{j}}[\varphi_{z_{i}z_{\bar{j}}}+\varphi_{z_{i}u}u_{\bar{j}}+\varphi_{uz_{\bar{j}}}u_{i}+\varphi_{uu}u_{\bar{j}}u_{i}+\varphi_{u}u_{i\bar{j}}]\\
= &\; D_{\alpha} \tilde{\psi} D_{\alpha}r+2F^{i\bar{j}}(D_{\alpha}u)_{i}(D_{\alpha}r)_{\bar{j}}+D_{\alpha}uF^{i\bar{j}}(D_{\alpha}r)_{i\bar{j}}\\
 & -F^{i\bar{j}}[\varphi_{z_{i}z_{\bar{j}}}+\varphi_{z_{i}u}u_{\bar{j}}+\varphi_{uz_{\bar{j}}}u_{i}+\varphi_{uu}u_{\bar{j}}u_{i}+\varphi_{u}u_{i\bar{j}}]\\
\ge &\; -C_{15}(1+M)\sum_{i=1}^{n}F^{i\bar{i}},
\end{align*}
where in the last inequality we used Theorem \ref{thm:second derivative boundary reduction}
and $C_{15}$ is a constant depending on $|u|_{C^1(\Omega)}$, $|\varphi|_{C^2}$ and
$|r|_{C^3(\Omega)}$.
Without loss of generality, we assume that
\[
M \ge 16(\sup_{\Omega}|Du|+\sup_{\Omega\times[\inf_{\Omega}u,\sup_{\Omega}u]}\varphi(z,u))^{2}.
\]
So we have
\[
|M^{-\frac{1}{2}}(<Du,Dr>-\varphi(z,u))|\le1/4.
\]

By Theorem \ref{thm:second derivative boundary reduction}
and $|D r|^{2}=1$ on $\Omega_{\mu}$, we obtain
\begin{align*}
0  \ge &\; F^{i\bar{j}}\Phi_{i\bar{j}}\\
  = &\; F^{i\bar{j}}[<Du,Dr>-\varphi(z,u)]_{i\bar{j}}\left(1+2M^{-\frac{1}{2}}(<Du,Dr>-\varphi(z,u))\right)\\
  & +  \frac{2M^{3/2}F^{i\bar{j}}r_{i}r_{\bar{j}}}{4\left(1+2M^{-\frac{1}{2}}(<Du,Dr>-\varphi(z,u))\right)^{2}}
 + \frac{1}{2}M F^{i\bar j} r_{i\bar{j}} \\
  \ge & \; -C_{15}(1+M)\sum_{i=1}^{n}F^{i\bar{i}}+\frac{ 2}{9}c_{n,k}\sum_{i=1}^{n}F^{i\bar{i}}M^{3/2}>0,
\end{align*}
for large $M$ satisfying $\frac{2}{9}c_{n,k}M^{3/2} > C_{15}(1+M)$.
This yields a contradiction. Hence, we can assume the maximum of $\Phi$ is achieved
on $\partial\Omega_{\mu}$ and therefore on $\partial\Omega.$

By Hopf lemma, we have on $\partial\Omega$
\begin{align*}
0 & \le\frac{\partial\Phi}{\partial\nu}\\
 & =(r_{l}D_{\nu}u_{l}+u_{l}D_{\nu}r_{l}-D_{\nu} \varphi ) \Big(1+M^{-\frac{1}{2}}(<Du,Dr> - \varphi ) \Big)+\frac{1}{2}M.
\end{align*}
If $\sup_{\partial\Omega}|u_{\nu\nu}|=-\inf_{\partial\Omega}u_{\nu\nu}=-u_{\nu\nu}(z_{1})=M,$ then
from the above inequality we have
\[ 0\leq  \frac{3}{4} u_{\nu\nu} (z_1) + C_{16} + \frac{1}{2}M \]
which implies that
 $\sup_{\partial\Omega}|u_{\nu\nu}|\le4C_{16}.$

Similarly, we can construct an auxiliary function
\[
\bar{\Phi} = \; <Du,Dr>-\varphi(z,u)-M^{-\frac{1}{2}}(<Du,Dr>-\varphi )^{2}-\frac{1}{2}Mr.
\]
With similar argument, we know $\bar{\Phi}$ achieves its minimum
at $\partial\Omega.$ On $\partial\Omega,$
\begin{align*}
0 & \ge\frac{\partial\bar{\Phi}}{\partial\nu}\\
 & =(r_{l}D_{\nu}u_{l}+u_{l}D_{\nu}r_{l}-D_{\nu} \varphi )\left(1-M^{-\frac{1}{2}}(<Du,Dr>-\varphi ) \right)-\frac{1}{2}M.
\end{align*}
If $\sup_{\partial\Omega}|u_{\nu\nu}|=\sup_{\partial\Omega}u_{\nu\nu}=u_{\nu\nu}(z_{2})=M$, then
from the above inequality we have
\[
0 \geq  \frac{3}{4}u_{\nu\nu} (z_2) -C_{17}-\frac{1}{2}M,
\]
which implies
$\sup_{\partial\Omega}|u_{\nu\nu}|\le 4C_{17}.$
\end{proof}

\begin{rmk}
The double normal estimates on boundary hold for general $\psi (z,u,Du)$ and thus the global $C^2$ estimates hold.
\end{rmk}

\begin{rmk}
The a priori estimates also hold for some general smooth symmetric functions of $n$ variables defined in a symmetric, open and convex cone $\Gamma\subset \mathbb{R}^{n}$ similar to those in \cite{CNS}.
\end{rmk}

\section{Proof of the main theorem}
\begin{proof}[Proof of Theorem \ref{thm}]
Now we can give the proof of Theorem \ref{thm}. The $C^0$ estimate is similar as \cite{CW} and we omit it here.
Combining our $C^1$, $C^2$ estimates and Evans-Krylov Theorem, we obtain
\[
|u|_{C^{2,\alpha} (\bar{\Omega})} \leq C
\]
for some uniform $C$ independent of the lower bound $\psi$ and $0<\alpha <1$.
Applying the method of continuity (see \cite{GT}, Theorem 17.28), we complete
the proof of Theorem \ref{thm}.
\end{proof}

\begin{proof}[Proof of Theorem \ref{thm2}]
The proof of Theorem 2  is almost the same as Theorem 1.3 in Chen-Wei \cite{CW} and we sketch the proof  here for the completeness. Denote $u_\varepsilon$ as the solution to
 \begin{equation}
\begin{cases}
\sigma_{k}(\Delta u \texttt{I} -\partial\bar{\partial}u)=\psi(z), & \mbox{\text{in}}\;\Omega,\\
u_\nu = - \varepsilon u + \phi(z), & \mbox{on}\;\partial\Omega.
\end{cases}
\end{equation}
The existence of $u_\varepsilon$ holds by Theorem \ref{thm}.
Following the proof of Theorem 1.3 in Chen-Wei \cite{CW} closely, we can obtain $|\varepsilon u_{\varepsilon}|_{C^0(\Omega)}\le C$, $|\nabla u_{\varepsilon }|\le C$, and $|\nabla^2 u_{\varepsilon }|\le C$, where $C$ is a positive constant independent of $\varepsilon$.  Here the convexity condition is to obtain the uniform $C^1$ estimates respect $\varepsilon$. 
Let $$v_{\varepsilon}=u_\varepsilon-\frac{1}{|\Omega|}\int_{\Omega }u_\varepsilon$$ and satisfy

$$\left\{\begin{array}{l}
\frac{\sigma_{n}\left(\Delta  v_{\varepsilon} \texttt{I} -\partial\bar{\partial} v_{\varepsilon}\right)}{\sigma_{l}\left(\Delta  v_{\varepsilon} \texttt{I} -\partial\bar{\partial} v_{\varepsilon}\right)}=\psi(z), \quad \text { in } \quad \Omega, \\
D_{\nu}\left(v_{\varepsilon}\right)=-\varepsilon v_{\varepsilon}-\frac{1}{|\Omega|} \int_{\Omega} \varepsilon u_{\varepsilon}+\phi(z), \quad \text { on } \quad \partial \Omega.
\end{array}\right.$$

By the uniform global $C^1$ estimates, there exist a subsequence $v_{\varepsilon}$ converging to $v$ and a constant $c$ such that 
$$\left\{\begin{array}{l}
\frac{\sigma_{n}\left(\Delta  v \texttt{I} -\partial\bar{\partial} v\right)}{\sigma_{l}\left(\Delta  v \texttt{I} -\partial\bar{\partial} v\right)}=\psi(z), \quad \text { in } \quad \Omega, \\
D_{\nu}\left(v\right)=c+\phi(z), \quad \text { on } \quad \partial \Omega.
\end{array}\right.$$
The uniqueness is obtained by the Hopf lemma and same as \cite{CW}.
\end{proof}

\begin{proof}[Proof of Theorem \ref{Hessian quotient estimates}]
The $C^0$ estimate holds by the standard argument in \cite{CW}. Then Theorem 3 follows by Theorem \ref{global gradient}, Theorem \ref{thm:second derivative boundary reduction} and Theorem \ref{thm:double normal}.
\end{proof}

\begin{proof}[Proof of Theorem \ref{thm-real part}]
The proof is same as Theorem \ref{thm} by deleting the bar.
\end{proof}

\end{document}